\documentclass[11pt]{article}
\usepackage{amsmath}
  \usepackage{paralist}
  \usepackage{graphics}
  \usepackage{epsfig}
  \usepackage{float}
  \usepackage{graphicx}
  \usepackage{abstract}
 \usepackage{epstopdf}
 \usepackage{amssymb}
 \usepackage{cite}
 \usepackage{mathrsfs}
 \usepackage{amsthm}
 \usepackage[colorlinks=true]{hyperref}
\hypersetup{urlcolor=blue, citecolor=red}

\numberwithin{equation}{section}

\newtheorem{theorem}{Theorem}[section]

\newtheorem{physical conclusion}{Physical Conclusion}

\newtheorem{definition}[theorem]{Definition}
\newtheorem{remark}[theorem]{Remark}

\title{ Magneto-hydrodynamical  Model for Plasma}
\author{ Ruikuan Liu\thanks{Email:liuruikuan2008@163.com. Supported by NSFC(11401479)
} \ \  Jiayan Yang\thanks{Corresponding author:jiayan\_{}1985@163.com;} 
\\ \footnotesize $^{*,\dag}$Department of Mathematics, \footnotesize Sichuan University
 \footnotesize Chengdu, Sichuan 610064, China
 \\ \footnotesize $^\dag$Department of Mathematics and Information Technology,~Southwest Medical University \\ \footnotesize Chengdu,
Sichuan 646000, China}

\begin{document}
\date{}
\maketitle
\begin{abstract}
 Basing on the Newton's second law and the Maxwell equations for the electromagnetic fields, we  establish a new  3D incompressible magneto-hydrodynamics(MHD) equations for the motion of plasma under the standard Coulomb gauge. By using the Galerkin method, we prove a global weak solution for this 3D new model.
\begin{center}
\textbf{\normalsize keywords}
\end{center}
 Maxwell equations, Plasma, Coulomb gauge, Magneto-hydrodynamics, Galerkin method.
\end{abstract}

\section{Introduction}

It is well known that magneto-hydrodynamics is the study of the dynamics for electrically conducting fluids which are frequently generated in nature and industry, for example, the sun, beneath the Earth's mantle, plasma, liquid metals, and so on. We refer to Bittencourt's monograph \cite{B} for the basic background and to Temam's article \cite{T} for related mathematical issues. The  model of magneto-hydrodynamics receives an increasing attentions from many scientists. Such as, S. Chandrasekhar \cite{C} first established magneto-hydrodynamic (MHD) equations. Afterwards, the extended magnetohydrodynamics (XMHD) model has been researched in high energy density (HED) plasma systems, see \cite{S},\cite{SL},\cite{YE} and their references.

For the MHD equations, Duvaut and Lions \cite{DL} proved the existence of global weak solutions in Leray energy space and the existence of classical solutions locally in time for smooth initial data. Until now, there have been  many studies on the problem of regularity of weak solutions for MHD equations. Xin \cite{HX} introduced the definition of interior suitable weak solutions and obtained partial regularity theorems. Kang and Lee \cite{KL} and Wang and Zhang \cite{WZ} gave other regularity criteria under the Ladyzhenskaya-Prodi-Serrin type conditions. Kang and Kim \cite{KK0},\cite{KK1} considered suitable weak solutions in the half space and gave a boundary regularity criteria.

Recently, the Hall magneto-hydrodynamic (Hall-MHD) model was established by  A. Marion, D. Pierre etc in \cite{MP}. The Hall-MHD model receives an increasing attentions
from plasma physicists. It is believed to be the key for understanding the problem of
magnetic reconnection. Indeed, space plasma observations provide strong evidence for the
existence of frequent and fast changes in the topology of magnetic field lines, associating to important events such as solar flares \cite{F}. For the existence of global weak and regular solutions of the (Hall-MHD) equations, we refer to \cite{FO},\cite{HG},\cite{LM}. For the generalized magnetohydrodynamics
(MHD) and Navier-Stokes systems, we refer to \cite{AS},\cite{F},\cite{LM},\cite{MB},\cite{S},\cite{SL},\cite{YA1},\cite{YA2}.

However, the corresponding models in above papers  are obtain by using  certain approximations
and assumptions. In this paper, basing on some basic physics principles,  a new 3D  incompressible magneto-hydrodynamic model for plasma shall be established without any assumptions. It is in particular that the model describes the un-static electronic field.  Moreover, the existence of the global weak solution of this 3D model is obtained by the
Galerkin method.

\vskip 3mm
This paper is organized as follows. In Section 2 we introduce the new MHD model and also show the new MHD model is compatible with Maxwell equations. Meanwhile, we will analysis the characters of our model and the classical MHD model. In Section 3, we  give some notations, definitions and also demonstrate that a global weak solution of the three dimensions incompressible MHD equations.

\section{A new MHD Model}
\subsection{ A new model}
The motion of plasma obeys the Maxwell  equations for electromagnetic fields
\begin{eqnarray}
\begin{aligned}
\label{21} \text{curl}\, \mathbf{E}=-\frac{\partial \mathbf{H}}{\partial t},
\end{aligned}
\\\begin{aligned}
\label{22} \text{div}\mathbf{H}=0,
\end{aligned}
\\\begin{aligned}
\label{23} \text{curl}\,\mathbf{H}=\mu_0(\mathbf{J}+\epsilon_0\frac{\partial\mathbf{E}}{\partial t}),
\end{aligned}
\\\begin{aligned}
\label{24} \text{div}\, \mathbf{E}=\frac{\rho}{\epsilon_0},
\end{aligned}
\end{eqnarray}
where $\rho,\mathbf{E},\mathbf{H},\mathbf{J},\epsilon_0$ and $\mu_0$ denote, respectively, the total charge density of the plasma, the electric field, the magnetic field, the current density, the electric permittivity, and the magnetic
permeability of free space.

Let $A_0$ be a scalar potential and $\mathbf{A}=(A_1,A_2,A_3)$ be a magnetic potential. Basing on the
mathematical theory of vector fields, we easily get the following equations from (\ref{21})-(\ref{24})
\begin{eqnarray}
\begin{aligned}
\label{25} \mathbf{H}=\text{rot}\mathbf{A},
\end{aligned}
\\\begin{aligned}
\label{26} \mathbf{E}=-\nabla \mathbf{A_0}-\frac{\partial \mathbf{A}}{\partial t},
\end{aligned}
\\\begin{aligned}
\label{27} \text{rot}\, \mathbf{H}=\mu_0(\mathbf{J}+\epsilon_0\frac{\partial\mathbf{E}}{\partial t}),
\end{aligned}
\\\begin{aligned}\label{28}
\text{div}\, \mathbf{E}=\frac{\rho}{\epsilon_0}.
\end{aligned}
\end{eqnarray}
Obviously, the equations (\ref{21})-(\ref{24}) are equivalent to  (\ref{25})-(\ref{28}).

In the following we will deduce new governing MHD equations according to
the classical Newton's second law and  the  Maxwell  equations (\ref{25})-(\ref{28}).

Firstly, from the Newton's second law, the motion of plasma are governed by the following Navier-Stokes equations
\begin{equation}\label{29}
  \frac{\partial \mathbf{u}}{\partial t}+(\mathbf{u}\cdot\nabla)\mathbf{u}=\nu\Delta \mathbf{u}-\frac{1}{\rho_0}\nabla p+\frac{1}{\rho_0}\mathbf{J}\times\mathbf{H}+f,
\end{equation}
where $\mathbf{u}=(u_1,u_2,u_3)$ is the velocity field, $p$ is the pressure, $\mathbf{J}\times \mathbf{H}$ is the Lorentz force, $\mathbf{J}$ is current density, $f$ is the external force, $\rho_0$ is the mass density and $\nu$ is the dynamic viscosity.

It is known that the plasma current density is  in direct proportion to the speed of plasma motion,
\begin{equation}\label{210}
  \mathbf{J}=\rho_e \mathbf{u},
\end{equation}
where the constant $\rho_e$ is  the equivalent charge density, which is a constant.

Together (\ref{29}) with (\ref{210}), we get
\begin{equation}\label{211}
  \frac{\partial \mathbf{u}}{\partial t}+(\mathbf{u}\cdot\nabla)\mathbf{u}=\nu\Delta\mathbf{u}-\frac{1}{\rho_0}\nabla p+\frac{\rho_e}{\rho_0}\mathbf{u}\times \mathbf{H}+f.
\end{equation}

Secondly, we shall obtain an equation of the motion for the magnetic
 field.
  From (\ref{27}),
    we have
    \begin{equation}\label{212}
  \epsilon_0\mu_0\frac{\partial \mathbf{E}}{\partial t}=\text{rot}\,\mathbf{H}-\mu_0 \mathbf{J}.
\end{equation}
Then, combining (\ref{26}) and (\ref{212}),  we obtain
\begin{equation}\label{213}
  \frac{\partial^2\mathbf{A}}{\partial t^2}=-\nabla \Phi-\frac{1}{\epsilon_0\mu_0}\text{rot}^2\mathbf{A}+\frac{1}{\epsilon_0}\mathbf{J},
\end{equation}
where $\Phi=\frac{\partial A_0}{\partial t}$.

As the matter of fact,  based on the Lorentz invariance of electromagnetism theory in \cite{MW}, the equations (\ref{25})-(\ref{28}) can be equivalently written as follows
\begin{equation}\label{214}
  \partial^\mu \mathbf{F}_{\mu\nu}=\mathbf{J}_\nu(\nu=0,1,2,3),
\end{equation}
 where $\mathbf{F}_{\mu\nu}=\frac{\partial \mathbf{A}_\nu}{\partial x^\mu}-\frac{\partial \mathbf{A}_\mu}{\partial x^\nu}$, $\mathbf{A}_\nu=(A_0,\mathbf{A})$ is the 4D electromagnetic potential,  $\mathbf{J}_\nu=(-c\rho,\mathbf{J})$ is the 4D current density. By the 4D current conservation law
 \begin{equation}\label{215}
  \partial^\mu \mathbf{J}_\mu=0
\end{equation}
 and the identity
 \begin{equation}\label{216}
  \partial^\nu\partial^\mu \mathbf{F}_{\mu\nu}=0,
\end{equation}
 we have that
 \begin{equation}\label{217}
  \partial^\nu(\partial^\mu \mathbf{F}_{\mu\nu}-\mathbf{J}_\nu)=0,
\end{equation}
which implies that the number of independent equations in (\ref{214}) are three and the number of unknown
functions of (\ref{214}) are four. Hence, we need to supplement a equation
, physically called the gauge-fixing equation, to enure a physical
solution. Basing on the physical fact, we take the Coulomb gauge
  \begin{equation}\label{218}
  \text{div}\mathbf{A}=0.
\end{equation}
By (\ref{218}), we obtain
 $$\text{rot}^2\mathbf{A}=-\Delta \mathbf{A}.$$
 Therefore,  (\ref{213}) can be rewritten  as follows
 \begin{equation}\label{219}
  \frac{\partial^2\mathbf{A}}{\partial t^2}=\frac{1}{\epsilon_0\mu_0}\Delta\mathbf{A}+\frac{1}{\epsilon_0}\mathbf{J}-\nabla \Phi.
\end{equation}
Finally, from (\ref{29}), (\ref{210}) and (\ref{219}), a  new 3D incompressible MHD system
is given by:
\begin{eqnarray}
\label{220} \left\{
   \begin{aligned}
 & \frac{\partial \mathbf{u}}{\partial t}+(\mathbf{u}\cdot\nabla)\mathbf{u}=\nu\Delta\mathbf{u}-\frac{1}{\rho_0}\nabla p+\frac{\rho_e}{\rho_0}\mathbf{u}\times \text{rot}\mathbf{A}+f,
 \\&\frac{\partial^2\mathbf{A}}{\partial t^2}+\nabla \Phi=\frac{1}{\epsilon_0\mu_0}\Delta\mathbf{A}+\frac{\rho_e}{\epsilon_0}\mathbf{u},
 \\&  \nabla\cdot\mathbf{u}=0,
 \\& \nabla\cdot\mathbf{A}=0,
\end{aligned}
\right.
\end{eqnarray}
where $\mathbf{u}=(u_1,u_2,u_3)$, $\mathbf{A}=(A_1,A_2,A_3)$, $p$, $f$,$\Phi$,$\nu$ $\rho_0$, $\rho_e$, $\epsilon_0$ and $\mu_0$ denote, respectively, the velocity field, the magnetic potential, the pressure, the  external force, the  magnetic pressure, the dynamic viscosity, the mass density, the equivalent charge density, the electric permittivity  and the magnetic permeability of free space.

\begin{remark}
It is obvious that we deduce the new incompressible MHD equations basing on  the basic principles without any assumptions. Naturally, this new MHD equations exactly discrible the motion of plasma. It is worth pointing out $\Phi=\frac{\partial A_0}{\partial t}$ represents the magnetic pressure. The Coulomb gauge (\ref{218}) guarantees  a physical solution.
\end{remark}

\begin{remark}
When the dimension  $N$ of the space is $2$, the operator curl
and $\widetilde{curl}$ are defined as follows
\begin{eqnarray*}
\begin{aligned}
  \text{curl} \mathbf{u}&=\frac{\partial u_2}{\partial x_1}-\frac{\partial u_1}{\partial x_2} \ \  \ \text{for\ \ each \ \ vector\ \ function}\ \ \mathbf{u}=(u_1,u_2),
  \\ \widetilde{curl}\Psi&=(\frac{\partial\Psi}{\partial x_1},-\frac{\partial \Psi}{\partial x_2})\ \  \ \text{for\ \ any \ \ scalar\ \ function}\ \ \Psi.
\end{aligned}
\end{eqnarray*}
It is easy to observe that the two-dimension formula is given by
$$
 \widetilde{\text{curl}}\text{curl}~\mathbf{u}=\text{grad}\ \text{div}\mathbf{u}-\Delta \mathbf{u}.
$$
Eventually, our new model is also true for the 2D MHD.
\end{remark}

\subsection{Compatible with Maxwell equation}
It is known that the second equation of (\ref{220}):
\begin{equation}\label{221}
  \frac{\partial^2\mathbf{A}}{\partial t^2}=\frac{1}{\epsilon_0\mu_0}\Delta\mathbf{A}+
  \frac{\rho_e}{\epsilon_0}\mathbf{u}-\nabla \Phi
\end{equation}
is derived from the equations (\ref{25})-(\ref{27}). The gauge-fixing equation (\ref{218}) is the Coulomb gauge.

Now, we will show that the model (\ref{220}) is compatible with the Maxwell equations
(\ref{25})-(\ref{28})(or equivalently (\ref{21})-(\ref{24})).
Namely, we need to prove that (\ref{221}) and (\ref{218}) are compatible with the equation (\ref{24}).

 To divergent  both sides of (\ref{221}), and by (\ref{218}), we get
\begin{equation}\label{222}
  \frac{\partial}{\partial t}(\Delta A_0)=0.
\end{equation}
On the other hand, by (\ref{28}),  we obtain
\begin{equation}\label{223}
  \frac{\partial}{\partial t}(\text{div} \mathbf{E})=0,
\end{equation}
because $\rho$ is a constant in the plasma. By (\ref{26}) and (\ref{218}), we deduce that
\begin{equation}\label{224}
  \frac{\partial}{\partial t}(\text{div} \mathbf{E})=-\frac{\partial}{\partial t}(\Delta A_0)=0,
\end{equation}
which implies (\ref{221}) and (\ref{218}) are compatible with the equation (\ref{24}).

\subsection{The classical magneto-hydrodynamics (MHD) equations}
The classical magneto-hydrodynamics(MHD) equations was established by\\ S. Chandrasekhar in \cite{C}. He  took the following assumption
\begin{equation}\label{225}
\frac{\partial \mathbf{E}}{\partial t}=0,
\end{equation}
then the equations (\ref{23}) can be rewritten as follows
\begin{equation}\label{226}
\text{curl}\,\mathbf{H}=\mu_0\mathbf{J}.
\end{equation}
 The current density $\mathbf{J}$ is described by Ohm's law
\begin{equation}\label{227}
\mathbf{J}=\sigma(\mathbf{E}+\mu\mathbf{u}\times\mathbf{H}),
\end{equation}
where $\sigma$ is the coefficient of electrical conductivity.

According to (\ref{227}), we get
\begin{equation}\label{228}
\mathbf{E}=\frac{1}{\sigma}\mathbf{J}-\mu\mathbf{u}\times\mathbf{H}.
\end{equation}
It follows that from (\ref{226}) and (\ref{228})
\begin{equation}\label{229}
\mathbf{E}=\frac{1}{\mu_0\sigma}\text{rot}\,\mathbf{H}-\mu\mathbf{u}\times\mathbf{H}.
\end{equation}
Inserting (\ref{229}) for $\mathbf{E}$ into equation (\ref{21}), we have
\begin{equation}\label{230}
\frac{\partial\mathbf{H}}{\partial t}-\text{rot}\,\mathbf{u}\times\mathbf{H}=
-\text{rot}\,(\frac{1}{\sigma\mu_o\mu}\text{rot}\mathbf{H}).
\end{equation}
 For simplicity, setting $\mu=\sigma\mu_o\mu$,   then (\ref{230}) takes the form
\begin{equation}\label{231}
\frac{\partial \mathbf{H}}{\partial t}+(\mathbf{u}\cdot\nabla)\mathbf{H}=\mu\Delta\mathbf{H}
 +\mathbf{H}\cdot\nabla\mathbf{u}.
\end{equation}

By (\ref{226}), we deduce  $\mathbf{J}=\frac{1}{\mu_0}\text{rot}\,\mathbf{H}$, then Lorentz force $\textit{L}$ is given by
\begin{equation}\label{232}
  \textit{L}=\mathbf{J}\times\mathbf{H}=\frac{1}{\mu_0}\text{rot}\, \mathbf{H}\times\mathbf{H}.
\end{equation}
Similar to (\ref{29}), for the incompressible fluid, the equation of motion take the form
\begin{equation}\label{233}
  \frac{\partial \mathbf{u}}{\partial t}+(\mathbf{u}\cdot\nabla)\mathbf{u}=\nu\Delta\mathbf{u}-\frac{1}{\rho_0}\nabla p+\frac{1}{\mu_0}\text{rot}\, \mathbf{H}\times\mathbf{H}+f.
\end{equation}
where $\mathbf{u}=(u_1,u_2,u_3)$ is the velocity field, $p$ is the pressure, $f$ is the external force, $\rho_0$ is the mass density and $\nu$ is the dynamic viscosity.

Combining the motion of fluid and the the energy conservation,  and ignoring the displacement current,  the classical  MHD equations
are given by (\ref{231})-(\ref{233}) as follows
\begin{eqnarray}
\label{234} \left\{
   \begin{aligned}
 & \frac{\partial \mathbf{u}}{\partial t}+(\mathbf{u}\cdot\nabla)\mathbf{u}=\nu\Delta\mathbf{u}-\frac{1}{\rho_0}\nabla p+\frac{1}{\mu_0}\text{rot}\, \mathbf{H}\times\mathbf{H}+f,
 \\&\frac{\partial \mathbf{H}}{\partial t}+(\mathbf{u}\cdot\nabla)\mathbf{H}=\mu\Delta\mathbf{H}
 +\mathbf{H}\cdot\nabla\mathbf{u},
 \\&  \nabla\cdot\mathbf{u}=0,
 \\& \nabla\cdot\mathbf{H}=0,
\end{aligned}
\right.
\end{eqnarray}
where $\mathbf{u}$, $\mathbf{H}$, $p$, $\nu$, $\mu$ and $\rho_0$ denote, respectively, the velocity field, the magnetic field, the pressure,  the dynamic viscosity, the resistivity and the mass density.

\begin{remark} It is obvious to see that the magnetic equations in MHD (\ref{220})
 are different from that in (\ref{234}). In classical model ,  the first equation (\ref{25}) of the Maxwell equations is applied to get the  magnetic equation in (\ref{234}).
 In other word, the classical model (\ref{234}) holds true for the static electronic field.
 However,  we use the third equation (\ref{27}) of the Maxwell equations to get the corresponding  magnetic equation in (\ref{220}). Moreover, from the subsection $\mathbf{2.2}$, we also know the model(\ref{220}) is compatible
 with the  equation (\ref{28}) of the Maxwell equations.
\end{remark}


\section{Existence of the global weak solution of the MHD}
Now, we study the new 3D incompressible MHD equations with external force $f(x)$
\begin{eqnarray}
\label{31} \left\{
   \begin{aligned}
 & \frac{\partial \mathbf{u}}{\partial t}+(\mathbf{u}\cdot\nabla)\mathbf{u}=\nu\Delta\mathbf{u}-\frac{1}{\rho_0}\nabla p+\frac{\rho_e}{\rho_0}\mathbf{u}\times \text{rot}\mathbf{A}+f(x),
 \\&\frac{\partial^2\mathbf{A}}{\partial t^2}=\frac{1}{\epsilon_0\mu_0}\Delta\mathbf{A}+\frac{\rho_e}{\epsilon_0}\mathbf{u}-\nabla \Phi,
 \\&  \nabla\cdot\mathbf{u}=0,
 \\& \nabla\cdot\mathbf{A}=0,
\end{aligned}
\right.
\end{eqnarray}
in $\Omega_T=\Omega\times[0,T]\subset \mathbb{R}^3\times[0,\infty)$. Here $\mathbf{u}(x,t)=(u_1,u_2,u_3)$ is the velocity field, $\mathbf{A}(x,t)=(A_1,A_2,A_3)$ is the magnetic potential, $p$ is the pressure, $\Phi=\frac{\partial A_0}{\partial t}$ is the magnetic pressure with the scalar electromagnetic potential $A_0$, $\nu_0$ is the dynamic viscosity, $\rho_0$ is the mass density, $\rho_e$ is the equivalent charge density, $\epsilon_0$ is the electric permittivity and $\mu_0$ is the magnetic permeability of free space. Hereafter, we only consider the  following boundary conditions
\begin{equation}\label{32}
  \mathbf{u}(t,x)=0, \ \ \ \mathbf{A}(t,x)=0,\ \ \ \text{on} \ \ \partial\Omega\times[0,T].
\end{equation}

The initial conditions are chosen  as follows
\begin{equation}\label{33}
  \mathbf{u}(0,x)=\phi(x),\ \ \mathbf{A}(0,x)=\psi(x),\ \ \mathbf{A}_t(0,x)=\eta(x).
\end{equation}

\subsection{Notations and definitions}
We give a few \emph{notations} and \emph{definitions} and then state our main result of this paper.
Let $\Omega\subset\mathbb{R}^3$ be a bounded domain. Let $H^\tau(\Omega)(\tau=1,2)$ be the usual Sobolev space on $\Omega$ with
the norm $||\cdot||_{H^\tau}$ and $L^2(\Omega)$ be the Hilbert space with the usual norm
$||\cdot||$. The space $H^1_0(\Omega)$  we mean that the completion of $C^\infty_0(\Omega)$ under the norm $||\cdot||_{H^1}$. If $\digamma$ is a Banach space, we denote by $L^p(0,T;\digamma)$ the Banach space of the $\digamma$-value functions defined
in the interval $(0,T)$ that are $L^p$-integrable.

We also consider the following spaces of divergence-free functions (see Temam \cite{ST})
\begin{eqnarray*}
\begin{aligned}
  X&=\{\mathbf{u}\in C^\infty_0(\Omega,\mathbb{R}^3)\  |\  \text{div}\mathbf{u}=0 \ \ \text{in} \  \Omega\},
  \\Y&= \text{the closure of}\  X \  \text{in}\  L^2(\Omega, \mathbb{R}^3)
  \\&=\{\mathbf{u}\in L^2(\Omega,\mathbb{R}^3) \ |\  \text{div}\mathbf{u}=0 \ \ \text{in} \  \Omega\},
  \\W&=\text{the closure of} \ X \  \text{in}\  H^1(\Omega, \mathbb{R}^3)
  \\&=\{\mathbf{u}\in H^1_0(\Omega,\mathbb{R}^3) \ |\  \text{div}\mathbf{u}=0 \ \ \text{in} \  \Omega\}.
\end{aligned}
\end{eqnarray*}
The space $L^2(\Omega, \mathbb{R}^3)$ has the Leray decomposition $L^2(\Omega, \mathbb{R}^3)=W\oplus W^\perp$, where
$$W^\perp=\{\mathbf{u}\in L^2(\Omega, \mathbb{R}^3)\ | \ \mathbf{u}=\nabla p,\,p\in H^1(\Omega)\}$$
(Leray decomposition).
Throughout the paper $P$ will denote the orthogonal projection from $L^2(\Omega, \mathbb{R}^3)$ into $W$. Then the operator $B:D(B)\hookrightarrow Y\rightarrow Y$
given by $B=-P\Delta$ with domain $D(B)=H^2(\Omega, \mathbb{R}^3)\cap W$ is called
the Stokes operator, which is positive definite and self-adjoint operator by the
relation
$$(B\omega,v)=(\nabla\omega,\nabla v),\ \text{for\ \ all\ \ } \omega\in D(B),\ v\in W.$$
The operator $B^{-1}$ is linear continuous from $Y$ into $D(B)$, and since the injection of $D(B)$ in  $Y$ is compact, $B^{-1}$ can be considered
as a compact operator in $Y$. As an operator in  $Y$, $B^{-1}$ is also self-adjoint. Hence it possesses a sequence of eigenfunctions $\{e_j\}^\infty_{j=1}$ which form an orthogonal basis of $Y$
$$Be_j=\lambda_je_j,\ \ \ e_j\in D(B), $$
and
$$0<\lambda_1\leq\lambda_2\leq\lambda_3\leq\cdots,\ \ \ \lambda_j\rightarrow\infty,\ \ \text{for}\ \ j\rightarrow\infty.$$

\begin{definition}
Suppose that $\phi,\eta\in Y$, $\psi\in W$. For  any  $T>0$, a vector function $(\mathbf{u}, \mathbf{A})$ is called a global weak solution of problem ({\ref{31}})-(\ref{33}) on $(0, T)\times\Omega$
 if it satisfies the following conditions:
\begin{enumerate}
  \item $\mathbf{u}\in L^2(0,T; W)\cap L^\infty(0,T; Y),$
  \item $\mathbf{A}\in L^\infty(0,T; W),\  \mathbf{A}_t\in L^\infty(0,T; Y),$
  \item For any function $\mathbf{v}\in X$, there hold
\begin{equation*}
 \begin{aligned}
\int_{\Omega}\mathbf{u}\cdot\mathbf{v}\text{d}x + \int_{0}^{t}\int_{\Omega}(\mathbf{u}\cdot\nabla)\mathbf{u}\cdot\mathbf{v}&+ \nu\nabla\mathbf{u}\cdot\nabla\mathbf{v}- \frac{\rho_e}{\rho_0}(\mathbf{u}\times \text{rot}\mathbf{A})\cdot\mathbf{v} \text{d}x\text{d}t\\
&=\int_{0}^{t}\int_{\Omega}f\cdot\mathbf{v}\text{d}x\text{d}t
+\int_{\Omega}\phi\cdot\mathbf{v}\text{d}x
  \end{aligned}
\end{equation*}
and
  \begin{equation*}
\int_{\Omega}\frac{\partial\mathbf{A}}{\partial t}\cdot\mathbf{v}\text{d}x + \int_{0}^{t}\int_{\Omega} \frac{1}{\epsilon_0\mu_0}\nabla\mathbf{u}\cdot\nabla\mathbf{v}
+\frac{\rho_e}{\epsilon_0}\mathbf{u}\cdot\mathbf{v} \text{d}x\text{d}t=\int_{\Omega}\eta\mathbf{v}\text{d}x.
\end{equation*}

\end{enumerate}
\end{definition}

\subsection{Existence theorem}
For the classical MHD equations, Duvaut and Lions \cite{DL} proved the existence of global weak solutions in Leray energy space and the existence of classical solutions locally in time for smooth initial data. Now, we state our main result as follows.

\begin{theorem}
Let the initial value $\phi,\eta\in Y$,~$\psi\in W$. If $f\in Y, \Phi\in L^2(0,T; H^1_0(\Omega))$,
then there exists a global weak solution for the problem (\ref{31})-(\ref{33}).
\end{theorem}

\begin{proof}
Based on the standard  Gerlinkin's method. Let $\{e_k\}^\infty_{k=1}$ be $W$ and $Y$ common orthogonal  basis.
We consider the finite dimensional subspaces $W_k=\text{span}\{e_1(x),e_2(x),...,e_k(x)\}$, $k\in \mathbb{N}$, the corresponding
orthogonal projections $P_k:W\rightarrow W_k$. The approximate solutions
\begin{eqnarray}\label{eq0}
\begin{aligned}
&\mathbf{u}^k(x,t)=\sum_{i=1}^{k}c_{ik}(t)e_{i}(x),
\\&\mathbf{A}^k(x,t)=\sum_{i=1}^{k}d_{ik}(t)e_{i}(x),
\end{aligned}
\end{eqnarray}
expanded in terms of eigenfunctions of Stokes operators. Then, the
coefficients $c_{ik}(t)$ and $d_{ik}(t)$ are found by requiring that $\mathbf{u}^k$ and
$\mathbf{A}^k$ satisfy the following equations:
\begin{align*}
   &\frac{\partial\mathbf{u}^k}{\partial t}+P_k(\mathbf{u}^k\cdot\nabla\mathbf{u}^k)
   =-\nu B\mathbf{u}^k +\frac{\rho_e}{\rho_0}P_k(\mathbf{u}^k\times \text{rot}\mathbf{A}^k)+P_kf,
\\
  & \frac{\partial^2\mathbf{A}^k}{\partial^2 t}=-\frac{1}{\epsilon_0\mu_0}B\mathbf{A}^k
  +\frac{\rho_e}{\epsilon_0}P_k(\mathbf{u})-P_k(\nabla\Phi),
  \\
   &\mathbf{u}^k(0)=P_k(\phi), \mathbf{A}^k(0)=P_k(\psi), \mathbf{A}^k_t(0)=P_k(\eta),
\end{align*}
which owns the weak form
\begin{align}
   \label{34}\nonumber&\int_{0}^{t}\int_{\Omega}\mathbf{u}^k_t\cdot\alpha
   +(\mathbf{u}^k\cdot\nabla\mathbf{u}^k)\cdot\alpha +\frac{1}{\rho_0}\nabla p\cdot\alpha\text{d}x\text{d}t
   \\&=\int_{0}^{t}\int_{\Omega}\nu\Delta\mathbf{u}^k\cdot\alpha
   +\frac{\rho_e}{\rho_0}(\mathbf{u}^k\times \text{rot}\mathbf{A}^k)\cdot\alpha +f\cdot\alpha \text{d}x\text{d}t, \ \    \ \ \forall \alpha\in W_k,
\\
\label{35}\nonumber & \int_{0}^{t}\int_{\Omega}\mathbf{A}^k_{tt}\cdot\beta
   +\nabla\Phi\cdot\beta \text{d}x\text{d}t
   \\&=\int_{0}^{t}\int_{\Omega}\frac{1}{\epsilon_0\mu_0}\Delta
   \mathbf{A}^k\cdot\beta+\frac{\rho_e}{\epsilon_0}\mathbf{u}^k\cdot\beta \text{d}x\text{d}t, \ \    \ \ \forall\beta\in W_k,
  \\\label{36}
  &\mathbf{u}^k(0)=\phi_k(x), \mathbf{A}^k(0)=\psi_k(x), \mathbf{A}^k_t(0)=\eta_k(x).
\end{align}
Setting $\alpha=\mathbf{u}^k$ in (\ref{34}).  Note that using integration by parts, we get
\begin{align*}
&\int_{\Omega}(\mathbf{u}^k\cdot\nabla\mathbf{u}^k)\cdot\mathbf{u}^k\text{d}x
=\frac{1}{2}\int_{\Omega}(\mathbf{u}^k\cdot\nabla)|\mathbf{u}^k|^2\text{d}x
=-\frac{1}{2}\int_{\Omega}(\nabla\cdot\mathbf{u}^k)|\mathbf{u}^k|^2\text{d}x=0,
\end{align*}
since $\nabla\cdot\mathbf{u}^k=0$. It is easy to see that $(\mathbf{u}^k\times \text{rot}\mathbf{A}^k)\cdot\mathbf{u}^k=0$. Hence, (\ref{34}) can be rewritten by
\begin{align}\label{37}
   \frac{1}{2}\int_{0}^{t}\frac{d}{\text{d}t}\int_{\Omega}|\mathbf{u}^k|^2 \text{d}x\text{d}t=-\nu\int_{0}^{t}\int_{\Omega}|\nabla\mathbf{u}^k|^2 \text{d}x\text{d}t+\int_{0}^{t}\int_{\Omega}f\mathbf{u}^k\text{d}x\text{d}t.
\end{align}
Taking $\beta=\mathbf{A}^k_t$ in (\ref{35}), we have
\begin{align}\label{38}
\frac{1}{2}\int_{0}^{t}\frac{d}{\text{d}t}\int_{\Omega}\big(|\mathbf{A}^k_t|^2 +\frac{1}{\epsilon_0\mu_0}
   |\nabla\mathbf{A}^k|^2\big)\text{d}x\text{d}t=\int_{0}^{t}\int_{\Omega}
   \frac{\rho_e}{\epsilon_0}\mathbf{u}^k\cdot\mathbf{A}^k_t \text{d}x\text{d}t,
\end{align}
since $\int_{\Omega}\nabla\Phi\cdot\mathbf{A}^k_t\text{d}x
=-\int_{\Omega}\Phi(\nabla\cdot\mathbf{A}^k_t)\text{d}x=0.$

Adding (\ref{37})
  and  (\ref{38}),  using the H$\ddot{o}$lder inequality, we obtain
\begin{align}
   \label{}\nonumber &||\mathbf{u}^k(t)||^2+||\mathbf{A}^k_t(t)||^2+2\int_{0}^{t}
   \int_{\Omega}\nu|\nabla\mathbf{u}^k|^2 \text{d}x\text{d}t +\int_{\Omega}\frac{1}{\epsilon_0\mu_0}|\nabla\mathbf{A}^k|^2 \text{d}x
\\
   \label{}\nonumber&\leq 2||\mathbf{u}^k(0)||^2+2||\mathbf{A}^k_t(0)||^2+\int_{\Omega}
   \frac{1}{\epsilon_0\mu_0}|\nabla\mathbf{A}^k(0)|^2 \text{d}x
  \\  \label{39} &
   +\bigg(\frac{\rho_e}{\epsilon_0}+2\bigg)\bigg(\int_{0}^{t}||\mathbf{u}^k(t)||^2\text{d}t
   +\int_{0}^{t}||\mathbf{A}^k_t(t)||^2\text{d}t\bigg)
 +2\int_{0}^{t}\int_{\Omega}f^2 \text{d}x\text{d}t.
\end{align}
Notice that
\begin{align}
\nonumber&||\mathbf{u}^k(0)||=||\phi^k||\leq||\phi||,\ \ ||\mathbf{A}^k_t(0)||=||\eta^k||\leq||\eta||,\\
\label{310}&||\nabla\mathbf{A}^k(0)||=||\nabla\psi^k||\leq||\nabla\phi||.
\end{align}
It follows that (\ref{39}) and (\ref{310}),  we have the following estimate by Gronwall inequality
\begin{align}
   \nonumber &||\mathbf{u}^k(t)||^2+||\mathbf{A}^k_t(t)||^2\leq 2(||\phi||+||\eta||+||\nabla\psi||+||f||)e^{(2+\rho_e/\epsilon_0)t},
\\
  \nonumber& ||\nabla\mathbf{A}^k(t)||^2\leq
  2(||\phi||+||\eta||+||\nabla\psi||+T||f||)e^{(2+\rho_e/\epsilon_0)t},
\\
  \label{311}&\int_{0}^{T}||\nabla\mathbf{u}^k||^2 \text{d}t \leq
  2(||\phi||+||\eta||+||\nabla\psi||+T||f||)e^{(2+\rho_e/\epsilon_0)t}.
\end{align}

Furthermore,  (\ref{311}) implies that the global existence in $t$ for approximations
$(\mathbf{u}^k,\mathbf{A}^k)$ and also that
\begin{align*}
&\{\mathbf{u}^k \}\ \text{is uniformly bounded in } L^2(0,T; W)\cap L^\infty(0,T; Y),\\
&\{\mathbf{A}^k\} \ \text{is uniformly bounded in } L^\infty(0,T; W),\\
&\{\mathbf{A}^k_t\} \ \text{is uniformly bounded in } L^\infty(0,T; Y).
\end{align*}
Therefore, we conclude that there exist $\mathbf{u}^0\in L^2(0,T; W)\cap L^\infty(0,T; Y)$, $\mathbf{A}^0\in L^2(0,T; V)\cap L^\infty(0,T; Y)$and $\mathbf{A}^0_{t}\in L^\infty(0,T; Y)$ and subsequence,which we still
denote by $\{\mathbf{u}^k \},\{\mathbf{A}^k\}$, $\{\mathbf{A}^k_t\} $ to simplify the notation, i.e.,
\begin{align}\label{500}
\left\{
\begin{aligned}
   & \mathbf{u}^k\rightharpoonup\mathbf{u}^0\  \text{weakly \ in} \ L^2(0,T; W) \ \text{and weak\  star\  in}\ L^\infty(0,T; Y),
\\
  & \mathbf{A}^k \rightharpoonup\mathbf{A}^0 \ \ \text{weak\  star\  in}\ W^{1,\infty}(0,T;Y)\cap L^\infty(0,T; W),
  \\
  & \mathbf{A}^k_t \rightharpoonup\mathbf{A}^0_{t}\ \text{ weak\  star\  in}\ L^\infty(0,T; Y),
\end{aligned}
\right.
\end{align}
for $0<T<\infty$.

Let $ \mathbf{u}^k \in L^2(0,T; W)\cap L^\infty(0,T; Y)$. We want to obtain a uniform bound
for $\frac{\text{d}\mathbf{u}^k}{\text{d}t}$.  For the equations
\begin{align}\label{eq}
\begin{aligned}
   &\frac{\text{d}\mathbf{u}^k}{\text{d}t}=-P_k(\mathbf{u}^k\cdot\nabla\mathbf{u}^k)
   -\nu B\mathbf{u}^k +\frac{\rho_e}{\rho_0}P_k(\mathbf{u}^k\times \text{rot}\mathbf{A}^k)+P_kf,
\end{aligned}
\end{align}
we need to show that each term on the right of (\ref{eq}) is uniform bounded.

For any $h(0<h<1)$ and $\mathbf{v}\in X$, it is easy to see that

\begin{align}
&\begin{aligned}
\bigg|\int_{t}^{t+h}\int_{\Omega}\nabla\mathbf{u}^k\cdot\nabla\mathbf{v}\text{d}x\text{d}t\bigg|
&\leq \bigg[\int_{0}^{T}\int_{\Omega} |\nabla\mathbf{u}^k|^2\text{d}x\text{d}t\bigg]^{\frac{1}{2}} \bigg[ \int_{t}^{t+h}\int_{\Omega} |\nabla\mathbf{v}|^2 \text{d}x\text{d}t\bigg]^{\frac{1}{2}}
\\ & \leq ||\nabla\mathbf{v}||_{L^2({\Omega},\mathbb{R}^3)}||\mathbf{u}^k||_{L^2(0,T;W)}h^{\frac{1}{2}},
\end{aligned}\\
&\begin{aligned}
\bigg|\int_{t}^{t+h}\int_{\Omega}(\mathbf{u}^k\cdot\nabla\mathbf{u}^k)\cdot\mathbf{v}
\text{d}x\text{d}t\bigg|
&\leq \sum\limits_{i,j}^{n}\int_{t}^{t+h}\int_{\Omega}\mathbf{u}^k_i\mathbf{u}^k_j\frac{\partial \mathbf{v}_j}{\partial x_i} \text{d}x\text{d}t\bigg|
\\ & \leq C||\mathbf{v}||_{C^1({\Omega},\mathbb{R}^3)}||\mathbf{u}^k||_{L^\infty(0,T;Y)}h,
\end{aligned}\\
&\begin{aligned}
\bigg|\int_{t}^{t+h}\int_{\Omega}(\mathbf{u}^k\times\text{rot}\mathbf{A}^k)
\cdot\mathbf{v}\text{d}x\text{d}t\bigg|
&\leq C\int_{t}^{t+h}\int_{\Omega} |\mathbf{u}^k||\nabla\mathbf{A}^k||\mathbf{v}| \text{d}x\text{d}t
\\ & \leq C||\mathbf{v}||_{C^1({\Omega},\mathbb{R}^3)}||\mathbf{u}^k||_{L^\infty(0,T;Y)}||\mathbf{A}^k||_{L^\infty(0,T;W)}h,
\end{aligned}\\
&\begin{aligned}
\bigg|\int_{t}^{t+h}\int_{\Omega}f\cdot\mathbf{v}\text{d}x\text{d}t\bigg|
&\leq ||f||_{L^2(\Omega)}||\mathbf{v}||_{L^2(\Omega,\mathbb{R}^3)}h,
\end{aligned}
\end{align}
Hence, from (\ref{eq0}) and (3.15)-(3.18), we have
\begin{equation*}
|c_{ik}(t+h)-c_{ik}(t)|\leq Ch^{\alpha}(0<\alpha\leq1),
\end{equation*}
where $C$ is independent on $k$. So, for any fixed $i$, ${c_{ik}}$ is uniformly bounded and equi-continuous in $t\in [0,T]$. According to  Arzela-Ascoli theorem, we have
\begin{equation}\label{eq3}
\lim\limits_{k\rightarrow\infty}\sup \limits_{0\leq t\leq T}\bigg[\int_{\Omega}(\mathbf{u}^k-\mathbf{u}^0)\mathbf{v}\text{d}x\bigg]^2=0, \ \ \ \text{for}\ \mathbf{v}\in X.
\end{equation}

Analogously, for any $v\in X$, we have
\begin{eqnarray}\label{50}
\begin{aligned}
\lim\limits_{k\rightarrow\infty}\int_{0}^{T}
\bigg[\int_{\Omega}(\mathbf{A}_t^k-\mathbf{A}^0_{t})\mathbf{v}\text{d}x\bigg]^2\text{d}t=0.\\
\end{aligned}
\end{eqnarray}
Thus, $\mathbf{u}^k\rightharpoonup\mathbf{u}^0$, $\mathbf{A}^k_t \rightharpoonup\mathbf{A}^0_{t}$ is uniformly.

From Lemma C.4.2 in \cite{MW0} and Theorem 5.2.1 in \cite{E1}, combining (\ref{500}) and (\ref{eq3})-(\ref{50}), we easily get \begin{align}\label{501}
\mathbf{u}^k\rightarrow\mathbf{u}^0,\  \mathbf{A}^k\rightarrow\mathbf{A}^0,\  \mathbf{A}^k_t\rightarrow\mathbf{A}^0_{t} \ \text{in}\ L^2((0,T)\times \Omega).
\end{align}

Note that for any $\mathbf{v} \in C^\infty(\Omega,\mathbb{R}^3)\cap W$, by (\ref{501}), we have
\begin{eqnarray}\label{502}
\begin{aligned}
\lim\limits_{k\rightarrow\infty}\int_{0}^{t}\int_{\Omega}(\mathbf{u}^k\cdot\nabla\mathbf{u}^k)\cdot\mathbf{v}
\text{d}x\text{d}t
&=-\lim\limits_{k\rightarrow\infty}\int_{0}^{t}\int_{\Omega}\sum\limits_{i,j}^{n}
\mathbf{u}^k_i\mathbf{u}^k_j\frac{\partial \mathbf{v}_i}{\partial x_j} \text{d}x\text{d}t
\\ &=-\int_{0}^{t}\int_{\Omega}
(\mathbf{u}_0\cdot\nabla\mathbf{v})\cdot \mathbf{u}_0 \text{d}x\text{d}t
\\ &=\int_{0}^{t}\int_{\Omega}
(\mathbf{u}_0\cdot\nabla\mathbf{u}_0)\cdot\mathbf{v}\text{d}x\text{d}t.
\end{aligned}
\end{eqnarray}
And from  Theorem 5.2.1 in \cite{E1},  we get
\begin{eqnarray*}
\begin{aligned}
\lim\limits_{k\rightarrow\infty}\int_{0}^{t}\int_{\Omega}
\mathbf{u}^k\times\text{rot}\mathbf{A}^k\cdot\mathbf{v}
\text{d}x\text{d}t
=\int_{0}^{t}\int_{\Omega}
\mathbf{u}_0\times\text{rot}\mathbf{A}_0\cdot\mathbf{v}
\text{d}x\text{d}t.
\end{aligned}
\end{eqnarray*}
So,   the problem (\ref{31})-(\ref{33}) exists a global weak solution.
\end{proof}

\end{document}